\documentclass[a4,11pt]{amsart}
\usepackage{amssymb,amsmath,amsthm,txfonts}
\usepackage{cite}
\usepackage{amsrefs}

\newtheorem{thm}{Theorem}[section]
\newtheorem{lem}[thm]{Lemma}

\newtheorem{prop}[thm]{Proposition}
\theoremstyle{definition}

\theoremstyle{remark}
\newtheorem{rem}[thm]{\textbf{Remark}}
\newtheorem{rems}[thm]{\textbf{Remarks}}

 \makeatletter
    
    \@addtoreset{equation}{section}
  \makeatother

\makeatletter
      \def\@makefnmark{%
         \leavevmode
            \raise.9ex\hbox{\check@mathfonts
                \fontsize\sf@size\z@\normalfont%
                            \@thefnmark}%
       }
      \makeatother

\newcommand{\p}{\mathbb{P}}
\renewcommand{\q}{\mathbb{Q}}

\newcommand{\R}{\mathbb{R}^{n}}
\newcommand{\D}{\textrm{div}}

\newcommand{\dd}{\textrm{d}}

\begin{document}

\title[]{Exterior Navier-Stokes flows for bounded data}
\author[]{Ken Abe}
\date{}
\address[K. ABE]{Department of Mathematics, Faculty of Science, Kyoto University, Kitashirakawa Oiwakecho, Sakyo, Kyoto 606-8502, Japan}
\email{kabe@math.kyoto-u.ac.jp}
\subjclass[2010]{35Q35, 35K90}
\keywords{Navier-Stokes equations, bounded function spaces, exterior problem, $D$-solutions}
\date{\today}

\maketitle


\begin{abstract}
We prove unique existence of mild solutions on $L^{\infty}_{\sigma}$ for the Navier-Stokes equations in an exterior domain in $\mathbb{R}^{n}$, $n\geq 2$, subject to the non-slip boundary condition.
\end{abstract}

\vspace{10pt}

\section{Introduction}

\vspace{10pt}

We consider the initial-boundary value problem of the Navier-Stokes equations in an exterior domain $\Omega\subset \R$, $n\geq 2$: 
\begin{equation*}
\begin{aligned}
\partial_t u-\Delta{u}+u\cdot \nabla u+\nabla{p}&= 0 \quad \textrm{in}\quad \Omega\times(0,T),    \\
\D\ u&=0\quad \textrm{in}\quad \Omega\times(0,T),    \\
u&=0\quad \textrm{on}\quad \partial\Omega \times (0,T),  \\
u&=u_0\quad\hspace{-4pt} \textrm{on}\quad \Omega \times \{t=0\}.  
\end{aligned}
\tag{1.1}
\end{equation*}
There is a large literature on the solvability of the exterior problem for initial data decaying at space infinity. However, a few results are available for non-decaying data. A typical example of non-decaying flow is a stationary solution of (1.1) having a finite Dirichlet integral, called $D$-solution \cite{Leray1933}. It is known that $D$-solutions are bounded in $\Omega$ and asymptotically constant as $|x|\to\infty$; see Remarks 1.2 (ii). In this paper, we do not impose on $u_0$ conditions at space infinity. 

The purpose of this paper is to establish a solvability of (1.1) for merely bounded initial data. We set the solenoidal $L^{\infty}$-space, 
\begin{equation*}
L^{\infty}_{\sigma}(\Omega)
=\left\{ f\in L^{\infty}(\Omega)\ \Bigg|\ \int_{\Omega}f\cdot \nabla \varphi\dd x=0 \quad \textrm{for}\ \varphi\in \hat{W}^{1,1}(\Omega) \right\},  
\end{equation*}
by the homogeneous Sobolev space $\hat{W}^{1,1}(\Omega)=\{\varphi\in L^{1}_{\textrm{loc}}(\Omega)\ |\ \nabla \varphi\in L^{1}(\Omega)\ \}$. For exterior domains, the space $L^{\infty}_{\sigma}$ agrees with the space of all bounded divergence-free vector fields, whose normal trace is vanishing on $\partial\Omega$ \cite{AG2}. The $L^{\infty}$-type solvability for (1.1) is recently established on $C_{0,\sigma}$ in the previous work of the author \cite{A4}, where $C_{0,\sigma}$ is the $L^{\infty}$-closure of $C_{c,\sigma}^{\infty}$, the space of all smooth solenoidal vector fields with compact support in $\Omega$. Since the condition $u_0\in C_{0,\sigma}$ imposes the decay $u_0\to 0$ as $|x|\to\infty$, we develop an existence theorem for non-decaying space $L^{\infty}_{\sigma}$, which in particular includes asymptotically constant vector fields. Moreover, the space $L^{\infty}_{\sigma}$ includes vector fields rotating at space infinity; see Remarks 1.2 (iv). When $\Omega$ is the whole space \cite{GIM} or a half space \cite{Sl03}, \cite{BJ}, the existence of mild solutions of (1.1) on $L^{\infty}_{\sigma}$ is proved by explicit formulas of the Stokes semigroup. In this paper, we prove unique existence of mild solutions on $L^{\infty}_{\sigma}$ for exterior domains based on $L^{\infty}$-estimates of the Stokes semigroup \cite{AG2}, \cite{A3}.

To state a result, let $S(t)$ denote the Stokes semigroup. It is proved in \cite{AG2} that $S(t)$ is an analytic semigroup on $L^{\infty}_{\sigma}$ for exterior domains of class $C^{3}$. Let $\p$ denote the Helmholtz projection. We write $\D\ F=(\sum_{i=1}^{n}\partial_{i}F_{ij})$ for matrix-valued functions $F=(F_{ij})$. It is proved in \cite{A3} that the composition operator $S(t)\p\D$ satisfies an estimate of the form
\begin{equation*}
\big\|S(t)\p\D F\big\|_{L^{\infty}(\Omega)}\leq \frac{C_{\alpha}}{t^{\frac{1-\alpha}{2}}}\big\|F\big\|_{L^{\infty}(\Omega)}^{1-\alpha}
\big\|\nabla F\big\|_{L^{\infty}(\Omega)}^{\alpha},    \tag{1.2}
\end{equation*}
for $F\in C^{1}_{0}\cap W^{1,2}(\Omega)$, $t\leq T_{0}$ and $\alpha\in (0,1)$. Here, $W^{1,2}(\Omega)$ denotes the Sobolev space and $C_{0}^{1}(\Omega)$ denotes the $W^{1,\infty}$-closure of $C_{c}^{\infty}(\Omega)$, the space of all smooth functions with compact support in $\Omega$. Although the projection $\p$ may not act as a bounded operator on $L^{\infty}$, the $L^{\infty}$-estimate (1.2) implies that the composition $S(t)\p\D$ is uniquely extendable to a bounded operator from $C_{0}^{1}$ to $C_{0,\sigma}$. Note that $F\in C^{1}_{0}$ imposes a decay condition at space infinity. Thus the extension to $C_{0}^{1}$ is not sufficient for studying non-decaying solutions. In this paper, we prove that the composition $S(t)\p\D$ is uniquely extendable to a bounded operator $\overline{S(t)\p\D}$ from the non-decaying space $W^{1,\infty}_{0}$ to $L^{\infty}_{\sigma}$, where $W^{1,\infty}_{0}$ is the space of all functions in  $W^{1,\infty}$ vanishing on $\partial\Omega$.

By means of the new extension, we study the integral equation on $L^{\infty}_{\sigma}$ of the form
\begin{equation*}
u(t)=S(t)u_{0}-\int_{0}^{t}\overline{S(t-s)\p \D} (uu)(s)ds.     \tag{1.3}
\end{equation*}
Here, $uu=(u_iu_j)$ is the tensor product. We call solutions of (1.3) mild solution on $L^{\infty}_{\sigma}$. Since the projection $\mathbb{P}$ may not be bounded on $L^{\infty}$, the extension  $\overline{S(t)\mathbb{P}\D}$ is not expressed by the individual operators. We thus prove that mild solutions satisfy  (1.1) by using a weak form. Let $C^{\infty}_{c,\sigma}(\Omega\times[0,T)) $ denote the space of all smooth solenoidal vector fields with compact support in $\Omega\times[0,T)$. Let $C([0,T]; X)$ (resp. $C_{w}([0,T]; X)$) denote the space of all (resp. weakly-star) continuous functions from $[0,T]$ to a Banach space $X$. Let $BUC_{\sigma}(\Omega)$ denote the space of all solenoidal vector fields in $BUC(\Omega)$ vanishing on $\partial\Omega$, where  $BUC(\Omega)$ is the space of all bounded uniformly continuous functions in $\overline{\Omega}$. Let $[\cdot]_{\Omega}^{(\beta)}$ denote the $\beta$-th H\"older semi-norm in $\overline{\Omega}$. The main result of this paper is the following: 

\vspace{10pt}

\begin{thm}
\noindent 
Let $\Omega$ be an exterior domain with $C^{3}$-boundary in $\mathbb{R}^{n}$, $n\geq 2$. For $u_0\in L^{\infty}_{\sigma}$, there exist $T\geq \varepsilon/||u_0||_{\infty}^{2}$ and a unique mild solution $u\in C_{w}([0,T]; L^{\infty})$ such that  
\begin{equation*}
\int_{0}^{T}\int_{\Omega}\big(u\cdot(\partial_{t}\varphi+\Delta \varphi)+u u: \nabla\varphi  \big)\dd x\dd t=-\int_{\Omega}u_0(x)\cdot \varphi (x,0)\dd x     \tag{1.4}
\end{equation*} 
for all $\varphi\in C^{\infty}_{c,\sigma}(\Omega\times [0,T))$, with some constant $\varepsilon=\varepsilon_{\Omega}$. The solution $u$ satisfies 
\begin{equation*}
\sup_{0< t\leq T}\left\{\big\|u\big\|_{L^{\infty}(\Omega)}(t)+t^{\frac{1}{2}}\big\|\nabla u\big\|_{L^{\infty}(\Omega)}(t)+t^{\frac{1+\beta}{2}}\Big[\nabla u\Big]^{(\beta)}_{\Omega}(t)  \right\}\leq C_1\big\|u_0\big\|_{L^{\infty}(\Omega)},       \tag{1.5}
\end{equation*}
\begin{equation*}
\sup_{x\in \Omega}\left\{\Big[u\Big]^{(\gamma)}_{[\delta,T]}(x)+\Big[\nabla u\Big]^{(\frac{\gamma}{2})}_{[\delta,T]}(x) \right\}\leq C_2\big\|u_0\big\|_{L^{\infty}(\Omega)},   \tag{1.6}
\end{equation*}
for $\beta, \gamma\in (0,1)$ and $\delta\in (0,T)$ with the constant $C_1$, independent of $u_0$ and $T$. The constant $C_2$  depends on $\gamma$, $\delta$ and $T$. If $u_0\in BUC_{\sigma}$, $u$, $t^{1/2}\nabla u\in C([0,T]; BUC)$ and $t^{1/2}\nabla u$ vanishes at time zero.
\end{thm}

\vspace{5pt}

\begin{rems}

\noindent
 (i) (Blow-up rate)
 By the estimate of the existence time in Theorem 1.1, we obtain a blow-up rate of mild solutions $u\in C_{w}([0,T_*); L^{\infty})$ of the form
\begin{equation*}
\|u\|_{L^{\infty}(\Omega)}\geq \frac{\varepsilon'}{\sqrt{T_*-t}}\quad \textrm{for}\ t<T_{*},
\end{equation*}
with $\varepsilon'=\varepsilon^{1/2}$, where $t=T_{*}$ is the blow-up time. The above blow-up estimate was first proved by Leray \cite{Leray1934} for $\Omega=\mathbb{R}^{3}$. See \cite{GIM} for $n\geq 3$ and \cite{Sl03} (\cite{Mar09}, \cite{BJ}) for a half space. The statement of Theorem 1.1 is valid also for a half space and improves regularity properties of mild solutions on $L^{\infty}_{\sigma}$ proved in \cite{Sl03}, \cite{BJ}.

\noindent
(ii)($D$-solutions) 
In \cite{Leray1933}, Leray proved the existence of $D$-solutions $u$ satisfying $u-u_{\infty}\in L^{6}(\Omega)$ for $u_{\infty}\in \mathbb{R}^{3}$ in the exterior domain $\Omega\subset \mathbb{R}^{3}$. His construction is based on an approximation for $R\to\infty$ of the problem
\begin{align*}
-\Delta u_{R}+u_{R}\cdot \nabla u_{R}+\nabla p_{R}&=0\quad \textrm{in}\ \Omega_R,\\
\D\ u_R&=0\quad \textrm{in}\ \Omega_R,\\
u_{R}&=0\quad \textrm{on}\ \partial\Omega,\\
u_{R}&=u_{\infty}\hspace{4pt} \textrm{on}\ \{|x|=R\},
\end{align*}
for $\Omega_{R}=\Omega\cap \{|x|<R\}$ (\cite[Chapter 5, Theorem 5]{Lady61}). See also \cite[Theorem 3.2]{Fujita61} (\cite[Theorem X.4.1]{Gal}) for a different construction. If the Dirichlet integral is finite, stationary solutions of (1.1) are locally bounded in $\overline{\Omega}$ (e.g.,  \cite[Theorem X.1.1]{Gal}). Moreover, $D$-solutions are bounded as $|x|\to\infty$ by $u-u_{\infty}\in L^{6}(\Omega)$. Thus, $D$-solutions are elements of $L^{\infty}_{\sigma}$ for $n=3$.

When $n=2$, more analysis is needed for information about the behavior as $|x|\to\infty$ since a finite Dirichlet integral does not imply decays at space infinity (e.g., $u=(\log |x|)^{\alpha}$ for $0<\alpha<1/2$). Leray's construction gives $D$-solutions also in $\Omega\subset \mathbb{R}^{2}$.
 It is proved in \cite{GW} (\cite{GW2}) that Leray's solutions are bounded in $\overline{\Omega}$ and converge to some constant $\overline{u}_{\infty}$ in the sense that $\int_{0}^{2\pi}|u(re_{r})-\overline{u}_{\infty}|\dd \theta\to 0$ as $r\to\infty$, where $(r,\theta)$ is the polar coordinate and $e_{r}=(\cos\theta,\sin\theta)$. Moreover, every $D$-solutions are bounded and asymptotically constant in the above sense \cite[Theorem 12]{Amick88}. Thus, $D$-solutions are elements of  $L^{\infty}_{\sigma}$ also for $n=2$. Theorem 1.1 yields a local solvability of (1.1) around $D$-solutions without imposing decay conditions for initial disturbance.

\noindent
(iii) (Global well-posedness for $n=2$) 
It is well known that the exterior problem (1.1) for $n=2$ is globally well-posed for initial data having finite energy, e.g., \cite{KO}. However, global well-posedness is unknown for non-decaying data $u_0\in L^{\infty}_{\sigma}$. For the whole space, the vorticity $\omega=\partial_{1}u^{2}-\partial_{2}u^{1}$ satisfies the a priori estimate 
\begin{align*}
||\omega||_{L^{\infty}(\mathbb{R}^{2})}\leq ||\omega_0||_{L^{\infty}(\mathbb{R}^{2})}\quad t>0.   
\end{align*}
It is proved in \cite{GMS} that the Cauchy problem of (1.1) for $n=2$ is globally well-posed for $u_0\in L^{\infty}_{\sigma}$ based on the local solvability result in \cite{GIM}. We proved a local solvability on $L^{\infty}_{\sigma}$ for exterior domains. Note that global solutions exist for rotationally symmetric initial data $u_0\in L^{\infty}_{\sigma}$; see below (iv).

\noindent 
(iv) (Rotating flows) An example of $u_0\in L^{\infty}_{\sigma}$ which is not asymptotically constant is a vector field rotating at space infinity. For example, we consider the two-dimensional unit disk $\Omega^{c}$ centered at the origin and a rotationally symmetric initial data $u_{0}=u_{0}^{\theta}(r)e_{\theta}(\theta)$ for $e_{\theta}(\theta)=(-\sin{\theta},\cos{\theta})$. Observe that $u_0$ is a  solenoidal vector field in $\Omega$ and a direction of $u_0$ varies for $\theta\in [0,2\pi]$ and $u_{0}^{\theta}\in L^{\infty}(1,\infty)$. Solutions of (1.1) for $u_0$ are rotationally symmetric and given by 
\begin{align*}
u=e^{t\Delta_{D}}u_0\qquad p=\int_{1}^{|x|}\frac{|u|^{2}}{r}\dd r,
\end{align*}
where $\Delta_{D}$ denotes the Laplace operator subject to the Dirichlet boundary condition. The solution $u$ is bounded in $\Omega\times (0,\infty)$ and non-decaying as $|x|\to\infty$.

\noindent 
(v) (Associated pressure) We invoke that the associated pressure of mild solutions on $L^{p}$ ($p\geq n$) is determined by the projection operator $\mathbb{Q}=I-\mathbb{P}$ and 
\begin{align*}
\nabla p=\mathbb{Q}\Delta u-\mathbb{Q}(u\cdot \nabla u).
\end{align*}
Since the projection $\mathbb{Q}$ may not be bounded on $L^{\infty}$, this representation is no longer available for mild solutions on $L^{\infty}_{\sigma}$. When $\Omega=\mathbb{R}^{n}$ or $\mathbb{R}^{n}_{+}$, the projection $\mathbb{Q}$ has explicit kernels and we are able to find associated pressure of mild solutions on $L^{\infty}$; see \cite{GIM} for $\Omega=\mathbb{R}^{n}$ and \cite{Sl03}, \cite{Mar09}, \cite{BJ} for $\Omega=\mathbb{R}^{n}_{+}$. Although explicit kernels are not available for exterior domains, we are able to find the associated pressure of mild solutions on $L^{\infty}$. We set 
\begin{align*}
\nabla p=\mathbb{K}W-\mathbb{Q}\D F   \tag{1.7}
\end{align*} 
for $W=-(\nabla u-\nabla^{T} u)n_{\Omega}$ and $F=uu$, where $n_{\Omega}$ is the unit outward normal on $\partial\Omega$ and $\mathbb{K}$ is a solution operator of the homogeneous Neumann problem (\textit{harmonic-pressure operator}) \cite[Remarks 4.3 (ii)]{AG2}. Note that $W=-\textrm{curl}\ u\times n_{\Omega}$ for $n=3$. The operators $\mathbb{K}$ and $\mathbb{Q}\D$ act for bounded functions and the associated pressure on $L^{\infty}$ is uniquely determined by (1.7) in the sense of distribution; see Remark 3.5 for a detailed discussion.   
\end{rems}

\vspace{10pt}
For asymptotically constant initial data $u_0$ (i.e., $u_{0}\to u_{\infty}$ as $|x|\to\infty$), local solvability of (1.1) for $n=3$ is proved in \cite[Theorem 5.2]{Miyakawa82} by means of the Oseen semigroup. In the paper, the problem (1.1) is reduced to an initial-boundary problem for decaying data by shifting $u$ by a constant $u_{\infty}$. Our analysis is based on the $L^{\infty}$-estimates of the Stokes semigroup which yields a local-in-time solvability of (1.1) without conditions for $u_0$ at space infinity.

The $L^{\infty}$-theory for the Cauchy problem of the Navier-Stokes equations is developed by Knightly \cite{K1}, \cite{K2}, Cannon and Knightly \cite{CK}, Cannone \cite{C} (\cite{CM}) and Giga et al. \cite{GIM}. For the whole space, mild solutions on $L^{\infty}$ are smooth and satisfy (1.1) in a classical sense \cite{GIM}. For a half space, mild solutions on $L^{\infty}$ are constructed in \cite{Sl03} (see also \cite{Mar09}, \cite{BJ}). There are a few results on solvability of the exterior problem for non-decaying data. In \cite{GMZ}, unique existence of continuous solutions of (1.1) for $n\geq 3$ is proved for non-decaying and H\"older continuous initial data. The result is extended in \cite{Mar2014} for merely bounded $u_0\in L^{\infty}_{\sigma}$ and $n\geq 3$ by using $L^{\infty}$-estimates of the Stokes semigroup \cite{AG1}, \cite{AG2}. Note that mild solutions on $L^{\infty}_{\sigma}$ are not constructed without the composition operator $\overline{S(t)\p\D}$. We proved the unique existence of mild solutions on $ L^{\infty}_{\sigma}$, which in particular yields a local-in-time solvability for $n=2$. The integral form (1.3) is fundamental for studying solutions of (1.1). We expect that mild solutions on $L^{\infty}$ are sufficiently smooth and satisfy (1.1) in a classical sense.\\

The article is organized as follows. In Section 2, we extend the composition operator $S(t)\p\D$ to a bounded operator from $W^{1,\infty}_{0}$ to $L^{\infty}_{\sigma}$ by approximation as we did the Stokes semigroup in \cite{AG2}. We extend $S(t)\p\D$ as a solution operator $F\longmapsto v(\cdot,t)$ for solutions $(v,q)$ of the Stokes equations for $v_0=\p\D\ F$. Note that $v_0=\p\D\ F$ for $F\in W^{1,\infty}_{0}$ is not an element of $L^{\infty}$ in general since the projection $\p$ is not bounded on $L^{\infty}$. We understand $v_0=\p\D\ F$ as distribution by using the fact that $\nabla \p\varphi\in L^{1}$ for $\varphi\in C_{c}^{\infty}$ (Lemma A.1). We approximate $F\in W^{1,\infty}_{0}$ by a sequence $\{F_m\}\subset C_{c}^{\infty}$ locally uniformly in $\overline{\Omega}$ and obtain a unique extension $\overline{S(t)\p\D}: F\longmapsto v(\cdot,t)$ by a limit $v$ of the sequence $v_m=S(t)\p\D\ F_m$.

In Section 3, we prove Theorem 1.1. We approximate initial data $u_0\in L^{\infty}_{\sigma}$ by a sequence $\{u_{0,m}\}\subset C_{c,\sigma}^{\infty}$ satisfying $u_{0,m}\to u_0$ a.e. in $\Omega$ and $||u_{0,m}||_{\infty}\leq C||u_0||_{\infty}$. Since the property of mild solutions (1.4) may not follow from a direct iteration argument on $L^{\infty}_{\sigma}$, we construct mild solutions by approximation. We apply an existence theorem on $C_{0,\sigma}$ \cite{A4} and construct a sequence of mild solutions $u_m\in C([0,T]; C_{0,\sigma})$ satisfying (1.4)-(1.6) for $u_{0,m}\in C_{0,\sigma}$. We prove that $u_m$ subsequently converges to a mild solution $u$ for $u_0\in L^{\infty}_{\sigma}$ locally uniformly in $\overline{\Omega}\times (0,T]$. 

In Appendix A, we show that $\nabla \p\varphi\in L^{1}$ for $\varphi\in C_{c}^{\infty}$ by means of the layer potential.

\vspace{10pt}

\section{An extension of the composition operator}

\vspace{5pt}

In this section, we prove that the composition operator $S(t)\p\partial$ is uniquely extendable to a bounded operator from $W^{1,\infty}_{0}$ to $L^{\infty}_{\sigma}$. We prove unique existence of solutions of the Stokes equations for initial data $v_0=\p\partial f$, $f\in W^{1,\infty}_{0}$, and extend the composition as a  solution operator $S(t)\p\partial:f\longmapsto v(\cdot,t)$. In what follows, $\partial=\partial_{j}$ indiscriminately denotes the spatial derivatives for $j=1,\cdots,n$.

\vspace{5pt}

\subsection{The Stokes system} We consider the Stokes equations,
\begin{equation*}
\begin{aligned}
\partial_{t}v-\Delta v+\nabla q&=0\quad \textrm{in}\ \Omega\times (0,T),   \\ 
\D\ v&=0\quad \textrm{in}\ \Omega\times (0,T), \\
v&=0\quad \textrm{on}\ \partial\Omega\times (0,T),   \\
v&=v_0\quad\hspace{-4pt} \textrm{on}\ \Omega\times \{t=0\}.
\end{aligned}
\tag{2.1}
\end{equation*}
We set the norm
\begin{equation*}
N(v,q)(x,t)=\bigl|v(x,t)\bigr|+t^{\frac{1}{2}}\bigl|\nabla v(x,t)\bigr|+t\bigl|\nabla^{2}v(x,t)\bigr|+t\bigl|\partial_{t}v(x,t)\bigr|+t\bigl|\nabla q(x,t)\bigr|.
\end{equation*} 
Let $d(x)$ denote the distance from $x\in \Omega$ to $\partial\Omega$. Let $(v, \nabla q)\in C^{2+\mu,1+\frac{\mu}{2}}(\overline{\Omega}\times (0,T])\times C^{\mu,\frac{\mu}{2}}(\overline{\Omega}\times (0,T]),\ \mu\in (0,1)$, satisfy the equations and the boundary condition of $(2.1)$. We say that $(v,q)$ is a solution of (2.1) for $v_0=\mathbb{P}\partial f$, $f\in W^{1,\infty}_{0}(\Omega)$, if  
\begin{equation*}
\sup_{0<t\leq T}\Bigg\{t^{\gamma}\big\|N(v,q)\big\|_{\infty}(t)+
t^{\gamma+\frac{1}{2}}\big\|d \nabla q\big\|_{\infty}(t)\Bigg\}<\infty, \tag{2.2}   
\end{equation*}
for some $\gamma\in [0,1/2)$ and 
\begin{equation*}
\int_{0}^{T}\int_{\Omega}\big(v\cdot(\partial_{t}\varphi+\Delta \varphi)-\nabla q\cdot \varphi\big)\textrm{d}x\textrm{d}t =\int_{\Omega}f\cdot \partial \mathbb{P}\varphi_{0}\dd x      \tag{2.3}
\end{equation*}
for all $\varphi\in C^{\infty}_{c}(\Omega\times [0,T))$ with $\varphi_{0}(x)=\varphi(x,0)$. The left-hand side is finite since $\varphi(\cdot,t)$ is supported in $\Omega$ and $\gamma<1/2$. The right-hand side is finite since $\partial \p\varphi_0$ is integrable in $\Omega$ for $\varphi_0\in C_{c}^{\infty}(\Omega)$ by Lemma A.1. As explained later in Remarks 2.9 (i), the operator $\mathbb{P}\partial$ is uniquely extendable for $f\in W^{1,\infty}_{0}$ and we are able to define $\mathbb{P}\partial f$ in the sense of distribution. The goal of this section is to prove:

\vspace{5pt}

\begin{thm}
Let $\Omega$ be an exterior domain with $C^{3}$-boundary. Let $T>0$. For $v_0=\p\partial f$, $f\in W^{1,\infty}_{0}(\Omega)$, there exists a unique solution $(v,q)$ of (2.1) satisfying 
\begin{equation*}
\sup_{0<t\leq T}\Big\{ t^{\gamma}\big\|N(v,q)\big\|_{\infty}(t)
+t^{\gamma+\frac{1}{2}}\big\|d \nabla q\big\|_{\infty}(t)\Big\}
\leq C\big\|f\big\|_{\infty}^{1-\alpha}\big\|f\big\|_{1,\infty}^{\alpha},     \tag{2.4}
\end{equation*}
for $\alpha\in (0,1)$ with $\gamma=(1-\alpha)/2$ and some constant C, depending on $\alpha$, $T$ and $\Omega$.
\end{thm}

\vspace{5pt}
Theorem 2.1 implies the following: 
\vspace{5pt}

\begin{thm}
The composition operator $S(t)\p\partial$ is uniquely extendable to a bounded operator $\overline{S(t)\p\partial}$ from $W^{1,\infty}_{0}(\Omega)$ to $L^{\infty}_{\sigma}(\Omega)$ together with the estimate
\begin{equation*}
\sup_{0< t\leq T}t^{\gamma+\frac{|k|}{2}+s}\Big\|\partial_{t}^{s}\partial_{x}^{k}\overline{S(t)\p \partial} f\Big\|_{\infty}\leq C\big\|f\big\|_{\infty}^{1-\alpha}\big\| f\big\|_{1,\infty}^{\alpha},  \tag{2.5}
\end{equation*}
for $f\in W^{1,\infty}_{0}(\Omega)$, $0\leq 2s+|k|\leq 2$ and $\alpha\in (0,1)$ with $\gamma=(1-\alpha)/2$.
\end{thm}

\vspace{10pt}

\subsection{H\"older estimates and uniqueness} In order to prove Theorem 2.1, we recall local H\"older estimates and a uniqueness result for the Stokes equations. In the subsequent section, we give a proof for Theorem 2.1 by approximation.\\
 
We set the H\"older semi-norm
\begin{equation*}
\Big[f\Big]^{(\mu,\frac{\mu}{2})}_{Q}=\sup_{t\in (0,T]}\Big[f\Big]^{(\mu)}_{\Omega}(t)+\sup_{x\in \Omega}\Big[f\Big]^{(\frac{\mu}{2})}_{(0,T]}(x),\quad \mu\in (0,1),
\end{equation*}
for $Q=\Omega\times(0,T]$. We set 
\begin{equation*}
N=\sup_{\delta\leq t\leq T}\big\|N(v,q)\big\|_{L^{\infty}(\Omega)}(t)
\end{equation*} 
for solutions $(v,q)$ of (2.1). The following local H\"older estimate is proved in \cite[Proposition 3.2 and Theorem 3.4]{AG1} based on the Schauder estimates for the Stokes equations \cite{Sl07} (\cite{Sl76}, \cite{Sl06}). 
  
\vspace{5pt}

\begin{prop}
Let $\Omega$ be an exterior domain with $C^3$-boundary.  

\noindent
(i) (Interior estimates) For $\mu\in(0,1)$, $\delta >0$, $T>0$, $R>0$, there exists a constant $C=C\bigl(\mu,\delta,T, R,d\bigr)$ such that 
\begin{equation}
\Big[\nabla^{2}v\Big]^{(\mu,\frac{\mu}{2})}_{Q'}+\Big[v_t\Big]^{(\mu,\frac{\mu}{2})}_{Q'}+\Big[\nabla q\Big]^{(\mu,\frac{\mu}{2})}_{Q'}\leq CN \tag{2.6}
\end{equation}
holds for all solutions $(v,q)$ of (2.1) for $Q'=B_{x_0}(R)\times(2\delta,T]$ and $x_0\in\Omega$ satisfying $\overline{B_{x_0}(R)}\subset\Omega$, where $d$ denotes the distance from $B_{x_0}(R)$ to $\partial\Omega$.

\noindent
(ii) (Estimates up to the boundary) There exists $R_0>0$ such that for $\mu\in(0,1),\ \delta>0$, $T>0$ and $R\leq R_{0}$, there exists a constant $C$ depending on $\mu$, $\delta$, $T$, $R$ and $C^{3}$-regularity of $\partial\Omega$ such that (2.6) holds for all solutions $(v, q)$ of (2.1) for $Q'=\Omega_{x_0,R}\times(2\delta,T]$ and $\Omega_{x_0,R}=B_{x_0}(R)\cap\Omega$, $x_0\in\partial\Omega$.
\end{prop}

\vspace{5pt}

We observe the uniqueness of solutions for (2.1). The uniqueness of the Stokes equations (2.1) for $v_0\in L^{\infty}_{\sigma}$ in an exterior domain is proved based on the uniqueness result in a half space \cite{Sl03} by a blow-up argument; see \cite[Lemma 2.12]{AG2}. In order to prove Theorem 2.1, we need a stronger uniqueness result since solutions of (2.1) for $v_0=\p\partial f$, $f\in W^{1,\infty}_{0}$, may not be bounded at $t=0$. The corresponding uniqueness result for a half space is recently proved in \cite[Theorem 5.1]{A3}. We deduce the result for exterior domains by the same blow-up argument as we did in \cite{AG2}.

\vspace{5pt}

\begin{prop}
Let $\Omega$ be an exterior domain with $C^{3}$-boundary. Let $(v,\nabla q)\in C^{2,1}(\overline{\Omega}\times (0,T])\times C(\overline{\Omega}\times (0,T])$ satisfy the equations and the  boundary condition of (2.1), and (2.2) for some $\gamma\in [0,1/2)$. Assume that 
\begin{equation*}
\int_{0}^{T}\int_{\Omega}\big(v\cdot(\partial_{t}\varphi+\Delta \varphi)-\nabla q\cdot \varphi\big)\textrm{d}x\textrm{d}t =0,   
\end{equation*}
for all $\varphi\in C^{\infty}_{c}(\Omega\times [0,T))$. Then, $v\equiv 0$ and $\nabla q\equiv 0$.
\end{prop}

\vspace{5pt}

\subsection{Approximation} We prove Theorem 2.1. We show existence of solutions for the Stokes equations (2.1) for $v_0=\p\partial f$, $f\in W^{1,\infty}_{0}$,  by approximation. We approximate $f\in W^{1,\infty}_{0}$ by elements of $C^{\infty}_{c}$ locally uniformly in $\overline{\Omega}$. 

\vspace{5pt}

\begin{lem}
Let $\Omega$ be an exterior domain with Lipschitz boundary. There exist constants $C_1$, $C_2$ such that for $f\in W^{1,\infty}_{0}(\Omega)$ there exists a sequence of functions $\{f_m\}_{m=1}^{\infty}\subset C_{c}^{\infty}(\Omega)$ such that 
\begin{equation*}
\begin{aligned}
&\big\|f_{m}\big\|_{\infty}\leq C_1\big\|f\big\|_{\infty}, \\
&\big\|\nabla f_{m}\big\|_{\infty}\leq C_2\big\|f\big\|_{1,\infty},  \\
&f_m\to f\quad \textrm{locally uniformly in}\ \overline{\Omega}\quad \textrm{as}\ m\to\infty.
\end{aligned}
\tag{2.7}
\end{equation*}
\end{lem}

\vspace{5pt}

The proof of Lemma 2.5 is reduced to the whole space and bounded domains.

\vspace{5pt}

\begin{prop}
The statement of Lemma 2.5 holds when $\Omega=\mathbb{R}^{n}$ with $C_1=1$.
\end{prop}

\begin{proof}
We cutoff the function $f\in W^{1,\infty}(\R)$. Let $\theta\in C^{\infty}_{c}[0,\infty)$ be a cut-off function satisfying $\theta\equiv 1$ in $[0,1]$, $\theta\equiv 0$ in $[2,\infty)$ and $0\leq \theta\leq 1$. We set $\theta_{m}(x)=\theta(|x|/m)$ for $m\geq 1$ so that $\theta_{m}\equiv 1$ for $|x|\leq m$ and $\theta_{m}\equiv 0$ for $|x|\geq 2m$. Then, $f_m=f\theta_{m}$ satisfies (2.7). 
\end{proof}

\vspace{5pt}

\begin{prop}
Let $\Omega$ be a bounded domain with Lipschitz boundary. There exists a constant $C_3$ such that for $f\in W^{1,\infty}_{0}(\Omega)$ there exists a sequence of functions $\{f_m\}_{m=1}^{\infty}\subset C^{\infty}_{c}(\Omega)$ such that 
\begin{equation*}
\begin{aligned}
&\big\|\nabla f_m\big\|_{\infty}
\leq C_3\big\| \nabla f\big\|_{\infty}\\
f_m&\to f\quad \textrm{uniformly in}\ \overline{\Omega}\quad \textrm{as}\ m\to\infty.
\end{aligned}
\tag{2.8}
\end{equation*}  
\end{prop}

\begin{proof}
We begin with the case when $\Omega$ is star-shaped, i.e., $\lambda\Omega_{x_0}\subset \overline{\Omega}$ for some $x_0\in \Omega$ and all $\lambda<1$, where $\lambda\Omega_{x_0}=\{x_0+\lambda(x-x_0)\ |\ x\in \Omega  \}$. We may assume $x_0=0\in \Omega$ and $\lambda\Omega\subset \overline{\Omega}$ by translation. 

For $f\in W^{1,\infty}_{0}(\Omega)$, we set 
\begin{equation*}
f_{\lambda}(x)=
\begin{cases}
&f\big(x/\lambda\big)\quad x\in \lambda\Omega,\\
&0\qquad\hspace{15pt} x\in \Omega\backslash \overline{\lambda\Omega} .
\end{cases}
\end{equation*}
Then, $f_{\lambda}$ is in $ W^{1,\infty}(\Omega)$ since $f$ is vanishing on $\partial\Omega$. It follows that  
\begin{equation*}
\big\|\nabla f_{\lambda}\big\|_{\infty}
\leq \frac{1}{\lambda}\big\|\nabla f\big\|_{\infty},
\end{equation*}
and $f_{\lambda}\to f$ uniformly in $\overline{\Omega}$ as $\lambda\to 1$. By a mollification of $f_{\lambda}$, we obtain a  sequence $\{f_m\}\subset C_{c}^{\infty}(\Omega)$ satisfying (2.8) with $C_3=2$. 

For general $\Omega$, we take an open covering $\{D_j\}_{j=1}^{N}$ so that $\overline{\Omega}\subset \cup_{j=1}^{N}D_j$ and $\Omega_{j}=\Omega\cap D_{j}$ is Lipschitz and star-shaped for some $x_j\in \Omega_{j}$ \cite[Lemma II 1.3]{Gal}. We take a partition of unity $\{\xi_{j}\}_{j=1}^{N}\subset C_{c}^{\infty}(\R)$ such that $\sum_{j=1}^{N}\xi_{j}=1$, $0\leq \xi_{j}\leq 1$, spt $\xi_{j}\subset \overline{D_j}$ and set 
\begin{equation*}
f=\sum_{j=1}^{N}f_j,\quad f_j=f\xi_j.
\end{equation*}
Since spt $f_j\subset \overline{\Omega}_{j}$, $\xi_{j}=0$ on $\partial D_{j}$ and $f=0$ on $\partial\Omega$, $f_{j}$ is in $W^{1,\infty}_{0}(\Omega_{j})$. Since $\Omega_j$ is star-shaped for some $x_j\in \Omega_{j}$, there exists $\{f_{j,m}\}\subset C_{c}^{\infty}(\Omega_j)$ satisfying (2.8) in $\Omega_j$ with $C_3=2$. We extend $f_{j,m}\in C_{c}^{\infty}(\Omega_j)$ to $\Omega\backslash \overline{\Omega_j}$ by the zero extension (still denoted by $f_{j,m}$) and set $f_m=\sum_{j=1}^{N}f_{j,m}$. Then, $f_m\in C_{c}^{\infty}(\Omega)$ converges to $f$ uniformly in $\overline{\Omega}$. We estimate 
\begin{equation*}
\big\|\nabla f_m\big\|_{L^{\infty}(\Omega)}
\leq \sum_{j=1}^{N}\big\|\nabla f_{j,m}\big\|_{L^{\infty}(\Omega_{j})}
\leq 2\sum_{j=1}^{N}\big\|\nabla f_{j}\big\|_{L^{\infty}(\Omega_{j})}.
\end{equation*}
Since $\nabla f_j=\nabla f\xi_j+f\nabla \xi_j$ and  
\begin{equation*}
\big\| f\big\|_{L^{\infty}(\Omega)}
\leq C_{p}\big\|\nabla f\big\|_{L^{\infty}(\Omega)},
\end{equation*}
by the Poincar\'e inequality (e.g., \cite[5.8.1 Theorem 1]{E}), we obtain 
\begin{equation*}
\big\|\nabla f_m\big\|_{L^{\infty}(\Omega)}
\leq C\big\|\nabla f\big\|_{L^{\infty}(\Omega)}.
\end{equation*} 
Thus, $\{f_m\}\subset C_{c}^{\infty}(\Omega)$ satisfies (2.8). The proof is complete.
\end{proof}

\begin{proof}[Proof of Lemma 2.5]
The assertion follows from Propositions 2.6 and 2.7.
\end{proof}

\vspace{5pt}

We recall the a priori estimate of $S(t)\mathbb{P}\partial$ for $f\in C_{c}^{\infty}(\Omega)$ \cite[Theorem 1.2]{A3}.

\vspace{5pt}
\begin{prop}
There exists a constant $C$ such that 
\begin{equation*}
\sup_{0\leq t\leq T}t^{\gamma+\frac{|k|}{2}+s}\Big\|\partial_{t}^{s}\partial_{x}^{k}S(t)\p \partial f\Big\|_{\infty}\leq C\big\|f\big\|_{\infty}^{1-\alpha}\big\| \nabla f\big\|_{\infty}^{\alpha}  \tag{2.9}
\end{equation*}
for $f\in C^{\infty}_{c}(\Omega)$, $0\leq 2s+|k|\leq 2$ and $\alpha\in (0,1)$ with $\gamma=(1-\alpha)/2$.
\end{prop}

\vspace{5pt}

\begin{proof}[Proof of Theorem 2.1]
For $f\in W^{1,\infty}_{0}$, we take a sequence $\{f_m\}\subset C^{\infty}_{c}$ satisfying (2.7). For $v_{0,m}=\p\partial f_m$, there exists a solution of the Stokes equations $(v_m,q_m)$ satisfying  
\begin{equation*}
\int_{0}^{T}\int_{\Omega}\big( v_{m}\cdot (\partial_{t}\varphi+\Delta \varphi)-\nabla q_m\cdot \varphi \big)\dd x\dd t
=\int_{\Omega}f_m\cdot \partial \p\varphi_{0}\dd x,
\end{equation*}
for $\varphi\in C^{\infty}_{c}\big(\Omega\times [0,T)\big)$. By (2.9) and (2.7), there exists a constant $C$ independent of $m\geq 1$ such that  
\begin{equation*}
\sup_{0\leq t\leq T}\Big\{ t^{\gamma}\big\|N(v_m,q_m)\big\|_{\infty}(t)+t^{\gamma+\frac{1}{2}}\big\|d\nabla q_{m}\big\|_{\infty}(t)\Big\}
\leq C\big\|f\big\|_{\infty}^{1-\alpha}\big\| f\big\|_{1,\infty}^{\alpha}.
\end{equation*}
We apply Proposition 2.3 and observe that there exists a subsequence of $(v_m,q_m)$ such that $(v_m,q_m)$ converges to a limit $(v,q)$ locally uniformly in $\overline{\Omega}\times (0,T]$ together with $\nabla v_m$, $\nabla^{2}v_m$, $\partial_{t}v_m$ and $\nabla q_{m}$. By sending $m\to\infty$, we obtain a solution $(v,q)$ of (2.1) for $v_0=\p\partial f$. By Proposition 2.4, the limit $(v,q)$ is unique. We proved the unique existence of solutions of (2.1) for $v_0=\p\partial f$ and $f\in W^{1,\infty}_{0}$ satisfying (2.4). The proof is now complete.
\end{proof}

\vspace{5pt}

\begin{rems}

\noindent 
(i) By the approximation (2.7) we are able to extend the operator $\mathbb{P}\partial$ for $f\in W^{1,\infty}_{0}$. We take a sequence $\{f_m\}\subset C_{c}^{\infty}$ satisfying (2.7) by Lemma 2.5 and observe that $v_{0,m}=\mathbb{P}\partial f_m$ satisfies 
\begin{align*}
(v_{0,m}, \varphi)=-(f_m, \partial \mathbb{P}\varphi)\quad \textrm{for}\ \varphi\in C^{\infty}_{c}(\Omega).
\end{align*}
Since $\partial \mathbb{P}\varphi\in L^{1}(\Omega)$ by Lemma A.1, the sequence $\{v_{0,m}\}$ converges to a limit $v_0$ in the distributional sense and the limit $v_0$ satisfies $(v_0,\varphi)=-(f,\partial \mathbb{P}\varphi)$. Since the limit $v_0$ is unique, the operator $\mathbb{P}\partial$ is uniquely extendable for $f\in W^{1,\infty}_{0}$.

\noindent 
(ii) We recall that for a sequence $\{v_{0,m}\}_{m=1}^{\infty}\subset L^{\infty}_{\sigma}$ satisfying 
\begin{align*}
||v_{0,m}||_{\infty}&\leq K_1,\\
v_{0,m}\to v_{0}&\quad \textrm{a.e.}\ \Omega, 
\end{align*}
with some constant $K_1$, there exists a subsequence such that $S(t)v_{0,m}$ converges to $S(t)v_0$ locally uniformly in $\overline{\Omega}\times (0,\infty)$ \cite{AG2}. From the proof of Theorem 2.1, we observe that for a sequence $\{f_m\}\subset W^{1,\infty}_{0}$ satisfying 
\begin{align*}
||f_m||_{1,\infty}&\leq K_2,\\
f_m\to f\quad &\textrm{locally uniformly in}\ \overline{\Omega},
\end{align*}
$\overline{S(t)\p\partial} f_m$ subsequently converges to $\overline{S(t)\p\partial} f$ locally uniformly in $\overline{\Omega}\times (0,\infty)$.

\noindent 
(iii) The extension $\overline{S(t)\p\partial}$ satisfies the property 
\begin{align*}
S(t)\overline{S(s)\p\partial} f=\overline{S(t+s)\p\partial} f
\end{align*}
for $t\geq 0$, $s>0$ and $f\in W^{1,\infty}_{0}$. In fact, this property holds for $f_m\in C_{c}^{\infty}$ satisfying (2.7). By choosing a subsequence, $v_m(\cdot,t)=S(t)\p\partial f_m$ converges to $v(\cdot,t)=S(t)\p\partial f$ locally uniformly in $\overline{\Omega}\times (0,\infty)$ as in the proof of Theorem 2.1. For fixed $s>0$, sending $m\to\infty$ implies 
\begin{align*}
S(t)S(s)\p\partial f_m=S(t)v_m(s) 
&\to S(t)v(s) \\
&=S(t)\overline{S(s)\p\partial}f\quad \textrm{locally uniformly in $\overline{\Omega}\times (0,\infty)$}.
\end{align*}     
Thus the property is inherited to $\overline{S(t)\p\partial} f$.
\end{rems}

\vspace{10pt}

\section{Mild solutions on $L^{\infty}_{\sigma}$}

\vspace{5pt}

We prove Theorem 1.1 by approximation. We show that a sequence of mild solutions $\{u_m\}$ subsequently converges to a limit $u$ locally uniformly in $\overline{\Omega}\times (0,T]$ by the $L^{\infty}$-estimates (1.5) and (1.6). Then, by an approximation argument for linear operators, we show that the limit $u$ satisfies the integral equation (1.3). We first recall the existence of mild solutions on $C_{0,\sigma}$ \cite[Theorem 1.1]{A4}

\vspace{5pt}

\begin{prop}
For $u_0\in C_{0,\sigma}$, there exist $T\geq \varepsilon_0/||u_0||_{\infty}^{2}$ and a unique mild solution $u\in C([0,T]; C_{0,\sigma})$ satisfying (1.3)-(1.6).
\end{prop}

\vspace{5pt}
We approximate $u_0\in L^{\infty}_{\sigma}$ by elements of $C_{c,\sigma}^{\infty}\subset C_{0,\sigma}$. We take a sequence $\{u_{0,m}\}_{m=1}^{\infty}\subset C_{c,\sigma}^{\infty}(\Omega)$ satisfying 
\begin{equation*}
\begin{aligned}
&\|u_{0,m}\|_{\infty}\leq C\|u_{0}\|_{\infty}\\
&u_{0,m}\to u_{0}\quad \textrm{a.e. in}\ \Omega,
\end{aligned}
\tag{3.1}
\end{equation*}
with some constant $C$, independent of $m\geq 1$ \cite[Lemma 5.1]{AG2}. We apply Proposition 3.1 and observe that there exists $T_m\geq \varepsilon_0/||u_{0,m}||_{\infty}^{2}$ and a unique mild solution $u_m\in C([0,T_m]; C_{0,\sigma})$ satisfying 
\begin{equation*}
\begin{aligned}
u_m(t)&=S(t)u_{0,_m}-\int_{0}^{t}\overline{S(t-s)\p \D}F_m(s)ds,\\
F_m&=u_mu_m.
\end{aligned}
 \tag{3.2}
\end{equation*}
Since $T_m$ is estimated from below by (3.1), we take $T\geq \varepsilon/||u_0||_{\infty}^{2}$ for $\varepsilon=\varepsilon_0 C^{-2}/2$ so that $T_m\geq T$ and $u_m\in C([0,T]; C_{0,\sigma})$ for $m\geq 1$.

\vspace{5pt}

\begin{prop}
There exists a subsequence such that $u_m$ converges to a limit $u$ locally uniformly in $\overline{\Omega}\times (0,T]$ together with $\nabla u_m$.
\end{prop}

\vspace{5pt}

\begin{proof}
It follows from (1.5), (1.6) and (3.1) that 
\begin{align*}
\sup_{0\leq t\leq T}\Big\{\|u_{m}\|_{\infty}(t)+t^{\frac{1}{2}}\|\nabla u_{m}\|_{\infty}(t)+t^{\frac{1+\beta}{2}}\big[\nabla u_{m}\big]_{\Omega}^{(\beta)}(t)\Big\}
&\leq C_1'\|u_0\|_{\infty},  \tag{3.3}  \\
\sup_{x\in\Omega}\Big\{\big[u_{m}\big]_{[\delta ,T]}^{(\gamma)}(x)
+\big[\nabla u_{m}\big]_{[\delta ,T]}^{(\frac{\gamma}{2})}(x)\Big\}
&\leq C_2'\|u_0\|_{\infty},  \tag{3.4}
\end{align*}
for $\beta,\gamma\in (0,1)$ and $\delta \in (0,T]$ with some constants $C_1'$ and $C_2'$, independent of $m\geq 1$. Since $u_m$ and $\nabla u_m$ are uniformly bounded and equi-continuous in $\overline{\Omega}\times [\delta,T]$, the assertion follows from the Ascoli-Arzel\`a theorem.
\end{proof}

\vspace{5pt}

\begin{prop}
The limit $u\in C_{w}([0,T]; L^{\infty})$ is a mild solution for $u_0\in L^{\infty}_{\sigma}$.
\end{prop}

\begin{proof}
We observe that the limit $u$ satisfies (1.4) by sending $m\to\infty$. The estimates (3.3) and (3.4) are inherited to $u$. We prove that $u$ satisfies the integral equation (1.3). By (3.1) and choosing a subsequence, $S(t)u_{0,m}$ converges to $S(t)u_0$ locally uniformly in $\overline{\Omega}\times (0,T]$ by Remarks 2.9 (ii). It follows from (3.3) and Proposition 3.2 that 
\begin{equation*}
\begin{aligned}
||F_m||_{\infty}&\leq K,\\
||\nabla F_{m}||_{\infty}&\leq \frac{2}{s^{\frac{1}{2}}}K,\\
F_m\to F\quad &\textrm{locally uniformly in}\ \overline{\Omega}\times (0,T]\ \textrm{as}\ m\to\infty,
\end{aligned}
\tag{3.5}
\end{equation*}
for $F=uu$ and $K=C_1'||u_0||_{\infty}$. By choosing a subsequence, we have
\begin{align*}
\overline{S(\eta)\p\D} F_m\to \overline{S(\eta)\p\D} F\quad  \textrm{locally uniformly in}\ \overline{\Omega}\times (0,T],
\end{align*} 
for each $s\in(0,t)$ as in Remarks 2.9 (ii). It follows from (3.5) and (2.5) that 
\begin{equation*}
\big\|\overline{S(t-s)\p\D} F_m\big\|_{\infty}\leq \frac{C}{(t-s)^{\frac{1-\alpha}{2}}}\Bigg(1+\frac{2}{s^{\frac{\alpha}{2}}}\Bigg)K^{2}
\end{equation*}
for $0<s< t$ and $\alpha\in (0,1)$. By the dominated convergence theorem, we have  
\begin{align*}
\int_{0}^{t}\overline{S(t-s)\p\D} F_m\dd s\to
\int_{0}^{t}\overline{S(t-s)\p\D} F\dd s\quad \textrm{locally uniformly in}\ \overline{\Omega}\times [0,T].
\end{align*}
Thus sending $m\to\infty$ implies that the limit $u$ is a mild solution for $u_0\in L^{\infty}_{\sigma}$. Since $S(t)u_0$ is weakly-star continuous on $L^{\infty}$ at $t=0$ \cite{AG2}, so is $u$.
\end{proof}

\vspace{5pt}

It remains to show continuity at $t=0$ for $u_0\in BUC_{\sigma}$.

\vspace{5pt}

\begin{prop}
For $u_0\in BUC_{\sigma}$, $S(t)u_0$, $t^{1/2}\nabla S(t)u_0\in C([0,T]; BUC)$ and $t^{1/2}||\nabla S(t)u_0||_{\infty}\to 0$ as $t\to 0$.
\end{prop}

\begin{proof}
Since $S(t)$ is a $C_0$-analytic semigroup on $BUC_{\sigma}$ \cite{AG2}, $S(t)u_0\in C([0,T]; BUC_{\sigma})$. Moreover,  $t^{1/2}\nabla S(t)u_0$ is continuous and bounded for $t\in (0,T]$ in $BUC$. We show that $t^{1/2}||\nabla S(t)u_0||_{\infty}\to 0$ as $t\to0$. 

We divide $u_0$ into two terms by using the Bogovski{\u\i} operator. For $u_0\in BUC_{\sigma}$, there exists $u_{0}^{1}\in C_{0,\sigma}$ with compact support in $\overline{\Omega}$ and $u_{0}^{2}\in BUC_{\sigma}$ supported away from $\partial\Omega$ such that $u_0=u_0^{1}+u_0^{2}$ (see \cite[Lemma 5.1]{AG2}). Let $A$ denote the Stokes operator and $D(A)$ denote the domain of $A$ in $BUC_{\sigma}$. Since $S(t)$ is a $C_0$-semigroup on $BUC_{\sigma}$, $D(A)$ is dense in $BUC_{\sigma}$. It follows from the resolvent estimate \cite[Theorem 1.3]{AGH} that 
\begin{align*}
||\nabla v||_{\infty}\leq C(||v||_{\infty}+||Av||_{\infty})\quad \textrm{for}\ v\in D(A).  \tag{3.5}
\end{align*}
We take an arbitrary $\epsilon>0$. For $u_{0}^{1}\in C_{0,\sigma}$, there exists $\{u_{0,m}^{1}\}\subset C_{c,\sigma}^{\infty}$ such that $||u_{0}^{1}-u_{0,m}^{1}||_{\infty}\leq \epsilon$ for $m\geq N^{1}_{\epsilon}$. We apply (3.5) and observe that 
\begin{align*}
t^{\frac{1}{2}}||\nabla S(t)u_{0,m}^{1}||_{\infty}
&\leq  t^{\frac{1}{2}}C\big(||S(t)u_{0,m}^{1}||_{\infty}+||S(t)Au_{0,m}^{1}||_{\infty}\big)\\
&\leq t^{\frac{1}{2}}C'\big(||u_{0,m}^{1}||_{\infty}+||Au_{0,m}^{1}||_{\infty}\big)
\to 0\quad \textrm{as}\ t\to0.
\end{align*}
We estimate
\begin{align*}
\overline{\lim_{t\to 0}}t^{\frac{1}{2}}||\nabla S(t)u_0^{1}||_{\infty}
&\leq \overline{\lim_{t\to 0}}\big(t^{\frac{1}{2}}||\nabla S(t)(u_0^{1}-u_{0,m}^{1})||_{\infty}
+t^{\frac{1}{2}}||\nabla S(t)u_{0,m}^{1}||_{\infty}\big) \\
&\leq C''\epsilon.
\end{align*}
We set $u_{0,m}^{2}=\eta_{\delta_m}*u_{0}^{2}$ by the mollifier $\eta_{\delta_{m}}$ so that $u_{0,m}^{2}$ is smooth in $\overline{\Omega}$ and $||u_{0}^{2}-u_{0,m}^{2}||_{\infty}\leq \epsilon$ for $m\geq N_{\epsilon}^{2}$. Since $u_{0,m}^{2}$ is supported away from $\partial\Omega$, we have $AS(t)u_{0,m}^{2}=S(t)\Delta u_{0,m}^{2}$ (see \cite[Proposition 6.1]{AG2}). By a similar way as for $u_{0}^{1}$, we estimate $\overline{\lim}_{t\to 0}t^{1/2}||\nabla S(t)u_0^{2}||_{\infty}\leq C''\epsilon$. We proved 
\begin{align*}
\overline{\lim}_{t\to 0}t^{\frac{1}{2}}||\nabla S(t)u_0||_{\infty}\leq 2C''\epsilon.
\end{align*}
Since $\epsilon>0$ is arbitrary, we proved $t^{1/2}||\nabla S(t)u_0||_{\infty}\to 0$ as $t\to 0$. 
\end{proof}

\vspace{5pt}
\begin{proof}[Proof of Theorem 1.1]
The assertion follows from Propositions 3.1-3.4. The proof is now complete.
\end{proof}

\vspace{5pt}

\begin{rem}
We set the associated pressure of mild solutions on $L^{\infty}$ by (1.7) and the harmonic-pressure operator $\mathbb{K}: L^{\infty}_{\textrm{tan}}(\partial\Omega)\longrightarrow L^{\infty}_{d}(\Omega)$, which is a solution operator of the homogeneous Neumann problem, 
\begin{align*}
\Delta q&=0\quad \textrm{in}\ \Omega,\\
\frac{\partial q}{\partial n}&=\D_{\partial\Omega}W\quad \textrm{on}\ \partial\Omega.
\end{align*}
Note that $\Delta u\cdot n=\D_{\partial\Omega}W$ by the divergence-free condition of $u$. Here, $L^{\infty}_{\textrm{tan}}(\partial\Omega)$ denotes the space of all bounded tangential vector fields on $\partial\Omega$ and $L^{\infty}_{d}(\Omega)$ is the space of all functions $f\in L^{1}_{\textrm{loc}}(\Omega)$ such that $df$ is bounded in $\Omega$ for $d(x)=\inf_{y\in \partial\Omega}|x-y|$, $x\in \Omega$. Since $W=-(\nabla u-\nabla^{T}u)n$ is bounded on $\partial\Omega$ for mild solutions on $L^{\infty}$, $\nabla q=\mathbb{K}W$ is defined as an element of $L^{\infty}_{d}$. Moreover, $\mathbb{Q}\D F$ is uniquely defined for $F=uu\in W^{1,\infty}_{0}$ as a distribution by Remarks 2.9 (i). Thus the associated pressure is defined by (1.7) for mild solutions on $L^{\infty}$.
\end{rem}

\vspace{10pt}

\section*{acknowledgements}
The author is grateful to the anonymous referees for their valuable comments. This work was partially supported by JSPS through the Grant-in-aid for Research Activity Start-up 15H06312 and Kyoto University Research Founds for Young Scientists (Start-up) FY2015.

\vspace{10pt}

\appendix

\section{$L^{1}$-estimates for the Neumann problem}

\vspace{5pt}
In Appendix A, we prove that $\nabla \p\varphi\in L^{1}(\Omega)$, $\varphi\in C_{c}^{\infty}(\Omega)$, for an exterior domain $\Omega$. We first estimate $L^{1}$-norms of solutions for the Poisson equation in $\mathbb{R}^{n}$ by using the heat semigroup. Then, we reduce the problem to the homogeneous Neumann problem and estimate solutions by a layer potential.  
\vspace{5pt}

\begin{lem}
Let $\Omega$ be an exterior domain with $C^{2}$-boundary in $\mathbb{R}^{n}$, $n\geq 2$. Then, $\nabla \mathbb{P}\varphi\in L^{1}(\Omega)$ for $\varphi\in C_{c}^{\infty}(\Omega)$.
\end{lem}

\vspace{5pt}

We set $\nabla \Phi=\q\varphi$ for $\q=I-\p$. It suffices to show that $\nabla^{2}\Phi$ is integrable in $\Omega$. We recall that the $\Phi$ solves the Neumann problem
\begin{equation*}
\begin{aligned}
\Delta \Phi=\D \varphi\quad \textrm{in}\ \Omega,\\
\frac{\partial \varphi}{\partial n}=0\quad \textrm{on}\ \partial\Omega.
\end{aligned}
\tag{A.1}
\end{equation*}
See \cite[Lemma III.1.2]{Gal}. We observe that $\Phi\in C^{2}(\Omega)\cap C^{1}(\overline{\Omega})$ by the elliptic regularity theory (e.g., \cite[Teor. 4.1]{LMp}) since $\varphi$ is smooth in $\Omega$ and the boundary is $C^{2}$. We may assume that $0\in \Omega^{c}$ by translation. We take $R>0$ such that $\Omega^{c}\subset B_{0}(R)$. Let $E$ denote the fundamental solution of the Laplace equation, i.e., $E(x)=C_{n}|x|^{-(n-2)} $ for $n\geq 3$ and $E(x)=-(2\pi)^{-1}\log{|x|}$ for $n=2$, where $C_{n}=(an(n-2))^{-1}$ and $a$ denotes the volume of $n$-dimensional unit ball. We first show that the statement of Lemma A.1 is valid for $\Omega=\mathbb{R}^{n}$. In the sequel, we do not distinguish $\varphi\in C_{c}^{\infty}(\Omega)$ and its zero extension to $\mathbb{R}^{n}\backslash \Omega$.

\vspace{5pt}

\begin{prop}
Set $h=E*\varphi$ and $\Phi_1=-\textrm{div}\ h$. Then, $\nabla^{3}h$ is integrable in $\mathbb{R}^{n}$. In particular, $\nabla^{2}\Phi_1\in L^{1}(\mathbb{R}^{n})$.  
\end{prop}

\begin{proof}
By using the heat semigroup, we transform $h$ into  
\begin{align*}
h=\int_{0}^{\infty}e^{t\Delta}\varphi \dd t.
\end{align*}
We divide $h$ into two terms and observe that
\begin{align*}
\partial^{3}_{x}h=\int_{0}^{1}\partial_{x} e^{t\Delta}\partial^{2}_{x}\varphi \dd t+\int_{1}^{\infty}\partial^{3}_{x}e^{t\Delta}\varphi \dd t,
\end{align*}
where $\partial_{x}=\partial_{x_j}$ indiscriminately denotes the spatial derivatives for $j=1,\cdots n$. We estimate 
\begin{align*}
||\partial^{3}_{x} h||_{L^{1}(\mathbb{R}^{n})}
&\lesssim\int_{0}^{1}\frac{1}{t^{1/2}}||\partial^{2}_{x} \varphi||_{L^{1}(\mathbb{R}^{n})} \dd t+\int_{1}^{\infty}\frac{1}{t^{3/2}}||\varphi||_{L^{1}(\mathbb{R}^{n})} \dd t\\
&\lesssim ||\partial^{2}_{x}\varphi||_{L^{1}(\mathbb{R}^{n})}+||\varphi||_{L^{1}(\mathbb{R}^{n})}.
\end{align*}
We proved $\nabla^{3}h\in L^{1}(\mathbb{R}^{n})$.
\end{proof}

\vspace{5pt}

We reduce (A.1) to the homogeneous Neumann problem 
\begin{equation*}
\begin{aligned}
-\Delta \Phi_2&=0\quad \textrm{in}\ \Omega,\\
\frac{\partial \Phi_2}{\partial n}&=g\quad \textrm{on}\ \partial \Omega.
\end{aligned}
\tag{A.2}
\end{equation*}
We write connected components of $\Omega$ by unbounded $\Omega_0$ and bounded $\Omega_1$, $\cdots$, $\Omega_N$, i.e., 
$\Omega=\Omega_{0}\cup (\cup_{j=1}^{N}\Omega_{j})$.

\vspace{5pt}

\begin{prop}
Set $\Phi_2=\Phi-\Phi_1$. Then, $\Phi_2\in C^{2}(\Omega)\cap C^{1}(\overline{\Omega})$ solves (A.2) for $g=\textrm{div}_{\partial \Omega}(An)$ and $A=\nabla h-\nabla^{T} h$. The function $g\in C(\partial\Omega)$ satisfies 
\begin{align*}
\int_{\partial\Omega_{j}}g \dd {\mathcal{H}}=0\quad\textrm{for}\ j=0,1,\cdots,N.  \tag{A.3}
\end{align*}
\end{prop}

\begin{proof}
We observe that $\Phi_2\in C^{2}(\Omega)\cap C^{1}(\overline{\Omega})$ satisfies $-\Delta \Phi_2=0$ in $\Omega$ and $\partial\Phi_2/\partial n=\partial (\D h)/\partial n$ on $\partial\Omega$. We take an arbitrary $\rho\in C_{c}^{\infty}(\mathbb{R}^{n})$. Since $An=(\sum_{1\leq j\leq n}(\partial_jh^{i}-\partial_ih^{j})n^{j})_{1\leq i\leq n}$ is a tangential vector field on $\partial\Omega$ (i.e., $An\cdot n=0$ on $\partial\Omega$), applying integration by parts yields 
\begin{align*}
\int_{\partial\Omega} g\rho\dd {\mathcal{H}}
&=\int_{\partial\Omega}\textrm{div}_{\partial\Omega}(An)\rho\dd {\mathcal{H}}\\
&=-\int_{\partial\Omega}(An)\cdot \nabla \rho\dd {\mathcal{H}}\\
&=-\int_{\partial\Omega}(\partial_jh^{i}-\partial_ih^{j})n^{j}\partial_i \rho \dd{\mathcal{H}}\\
&=-\int_{\partial\Omega}\partial_j h^{i}n^{i}\partial_i \rho  \dd{\mathcal{H}}+\int_{\partial\Omega}\partial_j h^{i}n^{i}\partial_j \rho  \dd{\mathcal{H}},
\end{align*}
where the symbol of summation is suppressed. By integration by parts, we have 
\begin{align*}
\int_{\partial\Omega}\partial_jh^{i}n^{j}\partial_i\rho  \dd{\mathcal{H}}
&=\int_{\partial\Omega}(\Delta h^{i}\partial_i\rho+\nabla h^{i}\cdot \nabla \partial_i \rho)  \dd x\\
&=\int_{\partial\Omega}(\Delta h^{i}\partial_i\rho-\nabla \D h\cdot \nabla  \rho)  \dd x
+\int_{\partial\Omega}\nabla h^{i}\cdot \nabla \rho n^{i}\dd {\mathcal{H}}.
\end{align*}
Since $-\Delta h=\varphi$ is supported in $\Omega$, it follows that 
\begin{align*}
\int_{\partial\Omega}g\rho{\mathcal{H}}
&=-\int_{\Omega}(\Delta h-\nabla \D\ h)\cdot \nabla \rho\dd x \\
&=-\int_{\Omega}(\Delta h\cdot \nabla \rho+\Delta \D\ h\rho)\dd x
+\int_{\partial\Omega}\frac{\partial}{\partial n}\D\ h\rho \dd {\mathcal{H}}\\
&=\int_{\partial\Omega}\frac{\partial}{\partial n}\D\ h\rho \dd {\mathcal{H}}.
\end{align*}
Since $\partial\Omega$ is $C^{2}$, $n$ is extendable to a $C^{1}$-function in a tubular neighborhood of $\partial\Omega$. Thus, $g$ is continuous on $\partial\Omega$. Since $\rho\in C_{c}^{\infty}(\mathbb{R}^{n})$ is arbitrary, we proved $\partial (\D\ h)/\partial n=g$ on $\partial\Omega$. Since $g$ is a surface-divergence form, by integration by parts, (A.3) follows. The proof is complete.    
\end{proof}

\vspace{5pt}

We estimate $\Phi_2$ by means of the layer potential.

\vspace{5pt}

\begin{prop}
(i) For $g\in C(\partial\Omega)$ satisfying (A.3), there exists a moment $h\in C(\partial\Omega)$ satisfying $\int_{\partial\Omega}h\dd{\mathcal{H}}=0$ and 
\begin{align*}
-g(x)=\frac{1}{2}h(x)+\int_{\partial\Omega}n(x)\cdot \nabla_x E(x-y)h(y)\dd {\mathcal{H}}(y)\quad x\in \partial\Omega.
\end{align*}
(ii) Set the single layer potential 
\begin{align*}
\tilde{\Phi}_{2}(x)=-\int_{\partial\Omega}E(x-y)h(y)\dd {\mathcal{H}}(y).
\end{align*}
Then, $\tilde{\Phi}_{2}$ is continuous in $\overline{\Omega}$. Moreover, the normal derivative $\partial_n \tilde{\Phi}_{2}$ exits and is continuous on $\partial\Omega$. The function $\tilde{\Phi}_{2}$ satisfies (A.2) and decays as $|x|\to \infty$. 
\end{prop} 

\begin{proof}
The assertion (i) is based on the Fredholm's theorem. See \cite[(3.40), (3.13), (3.30)]{Foll}. Since $h$ is bounded on $\partial\Omega$, $\tilde{\Phi}_{2}$ is continuous in $\overline{\Omega}$. Moreover, we have  
\begin{align*}
-\frac{\partial \tilde{\Phi}_{2}}{\partial n}(x)
=\frac{1}{2}h(x)+\int_{\partial\Omega}n(x)\cdot \nabla E(x-y)h(y)\dd {\mathcal{H}}(y)\quad x\in \partial\Omega.
\end{align*}
See \cite[(3.25), (3.28)]{Foll}. Thus $\tilde{\Phi}_{2}$ satisfies (A.2) by the assertion (i). When $n\geq 3$, $\tilde{\Phi}_{2}(x)\to 0$ as $|x|\to\infty$ since the fundamental solution decays as $|x|\to\infty$. Moreover, when $n=2$, the average of $h$ on $\partial\Omega$ is zero and we have 
\begin{align*}
\tilde{\Phi}_{2}(x)
&=-\int_{\partial\Omega}(E(x-y)-E(x))h(y)\dd {\mathcal{H}}(y)\\
&=\frac{1}{2\pi}\int_{\partial\Omega}\log{\Bigg(\frac{|x-y|}{|x|}\Bigg)}h(y)\dd \mathcal{H}(y)\to 0\quad \textrm{as}\ |x|\to\infty.
\end{align*}
The proof is complete.
\end{proof}

\vspace{5pt}

\begin{prop}
The function $\tilde{\Phi}_2$ agrees with ${\Phi}_2$ up to constant.
\end{prop}

\begin{proof}
Since $\nabla \Phi_2=\nabla \Phi-\nabla \Phi_1$ is $L^{p}$-integrable in $\Omega$ for all $p\in (1,\infty)$ (e.g., \cite{SiSo}), we may assume that $\Phi_2\to 0$ as $|x|\to\infty$ by shifting $\Phi_2$ by a constant. We set $\Psi=\Phi_2-\tilde{\Phi}_2 $ and observe that $\Psi$ is continuous in $\overline{\Omega}$. Moreover, the normal derivative exists and is continuous on $\partial\Omega$ by Proposition A.4. The function $\Psi$ satisfies $-\Delta\Psi=0$ in $\Omega$, $\partial\Psi/\partial n=0$ on $\partial\Omega$ and $\Psi\to 0$ as $|x|\to\infty$. By the elliptic regularity theory \cite{LMp}, $\Psi$ is smooth in $\Omega$ and continuously differentiable in $\overline{\Omega}$.

We shall show that $\Psi\equiv 0$. Since $\Psi$ decays as $|x|\to\infty$, there exits a point $x_0\in \overline{\Omega}$ such that $\sup_{x\in \Omega}\Psi(x)=\Psi(x_0)$. Suppose that $x_0\in \partial\Omega$. Since the boundary of class $C^{2}$ satisfies the interior sphere condition, the Hopf's lemma \cite[Chapter 2 Theorem 7]{PW} implies that $\partial \Psi(x_0)/\partial n>0$. Thus $x_0\in \Omega$. We apply the strong maximum principle \cite[Chapter 2 Theorem 5]{PW} and conclude that $\Psi$ is constant. Since $\Psi$ decays as $|x|\to \infty$, we have $\Psi\equiv 0$. The proof is complete. 
\end{proof}

\vspace{5pt}

\begin{prop}
$\nabla^{2}\Phi_2$ is integrable in $\Omega$.
\end{prop}

\begin{proof}
Since $\nabla^{2}\Phi_2$ is integrable near the boundary $\partial\Omega$, it suffices to show that $\nabla^{2}\Phi_2\in L^{1}(\{|x|\geq 2R\})$. Since $h\in C(\partial\Omega)$ satisfies $\int_{\partial\Omega}h\dd {\mathcal{H}}=0$, we observe that 
\begin{align*}
\tilde{\Phi}_{2}(x)
&=-\int_{\partial\Omega}(E(x-y)-E(x))h(y)\dd {\mathcal{H}}(y)\\
&=\int_{0}^{1}\dd t\int_{\partial\Omega}y\cdot (\nabla E)(x-ty)h(y)\dd {\mathcal{H}}(y).
\end{align*}
Since $\Omega^{c}\subset B_{0}(R)$, for $|x|\geq 2R$ we observe that 
\begin{align*}
|x-ty|&\geq \big|\ |x|-|ty|\ \big| \\
&\geq |x|-R \\
&\geq \frac{|x|}{2}.
\end{align*}
Since $\tilde{\Phi}_{2}$ agrees with $\Phi_2$ up to constant, we estimate
\begin{align*}
|\nabla^{2}\Phi_2(x)|
&\lesssim \int_{0}^{1}\dd t\int_{\partial\Omega}\frac{|h(y)|}{|x-ty|^{n+1}}\dd {\mathcal{H}}(y)\\
&\lesssim \frac{1}{|x|^{n+1}}||h||_{L^{1}(\partial\Omega)}.
\end{align*}
Thus, $\nabla^{2}\Phi_2$ is integrable in $\{|x|\geq2R\}$. The proof is complete.
\end{proof}

\begin{proof}[Proof of Lemma A.1]
By Propositions A.2 and A.6, the assertion follows.
\end{proof}

\vspace{5pt}

\bibliographystyle{plain}
\bibliography{ref}

\end{document}